\documentclass[a4paper,12pt]{amsart}
\usepackage[italian,english]{babel}

\usepackage{geometry}
\usepackage{enumitem}
\usepackage{mathrsfs}
\usepackage{comment}
\usepackage[colorlinks, linkcolor=blue,anchorcolor=Periwinkle,
citecolor=blue,urlcolor=Emerald]{hyperref}

\usepackage{tikz}
\usetikzlibrary{arrows.meta}

\usepackage{amsmath}
\usepackage{amsthm}
\usepackage[T1]{fontenc}
\usepackage[utf8]{inputenc}
\usepackage{amssymb}
\usepackage{csquotes}
\usepackage{leftidx}

\usepackage{mathrsfs}
\usepackage[all,cmtip]{xy}
\usepackage{graphicx}
\usepackage{tikz-cd}
\tikzset{dynkdot/.style={circle,draw,scale=.38}}

\newtheorem{prop}{Proposition}[section]
\newtheorem{lem}[prop]{Lemma}
\newtheorem{teo}[prop]{Theorem}
\newtheorem{cor}[prop]{Corollary}

\theoremstyle{definition}
\newtheorem{de}[prop]{Definition}

\newtheorem{ex}[prop]{Example}

\newcommand{\sD}{{\mathcal D}}
\newcommand{\sE}{{\mathcal E}}

\newcommand{\sH}{{\mathcal H}}

\newcommand{\sM}{{\mathcal M}}

\newcommand{\ssM}{{\mathfrak M}}

\newcommand{\D}{{\mathbb D}}

\newcommand{\F}{{\mathbb F}}

\newcommand{\R}{{\mathbb R}}

\newcommand{\Z}{{\mathbb Z}}

\newcommand{\Hom}{{\mathrm{Hom}}}
\newcommand{\Ext}{{\mathrm{Ext}}}

\newcommand{\kQ}{\mathbf{k}Q}

{\left\lbrace\begin{array}{@{}l@{}}}%
{\end{array}\right.}

\title[Degenerations in graded quiver varieties and in derived categories]{Degenerations in graded quiver varieties\\ and in  derived categories 
of Dynkin quivers}


\usepackage{comment}
\usepackage[colorlinks, linkcolor=blue,anchorcolor=Periwinkle,
citecolor=blue,urlcolor=Emerald]{hyperref}

\newcommand\Fang[1]{{\color{red} \sf Fang: [\% #1]}}

\author{Alessandro Contu}
\address[Alessandro Contu]{Research Institute for Mathematical Sciences, 
Kyoto University, 
Kyoto 606-8502 JAPAN
}
\email{contu.alessandro.63z@st.kyoto-u.ac.jp}
\author{Fang Yang}
\address[Fang Yang]{Max-Planck Institute for Mathematics, Vivatsgasse 7 , Bonn 53111, Germany}
\email{f.yang@mpim-bonn.mpg.de}

\begin{document}
\maketitle
\begin{abstract}
For any acyclic quiver, Keller--Scherotzke provided a stratifying functor from the category of finite-dimensional modules of the singular Nakajima category to the derived category of the quiver. Under this functor, a degeneration of strata of a graded quiver variety corresponds to a degeneration, in the sense of Jensen–Su–Zimmermann, in the derived category. In this article, for any Dynkin quiver, we further investigate Jensen–Su–Zimmermann's partial order and show that any degeneration of objects in the derived category can be obtained in this way.
\end{abstract}

\tableofcontents
\section{Introduction}
Let $\mathfrak{g}$ be a finite-dimensional complex simple Lie algebra of $ADE$ type and let $U_q(L\mathfrak{g})$ be the associated quantum loop algebra. In the study of the representation theory of $U_q(L\mathfrak{g})$, a preeminent role has been played by Nakajima's graded quiver varieties, introduced in \cite{Nakajima2001}. In more detail, let $Q$ be a quiver of the same Dynkin type as $\mathfrak{g}$ and let $I$ be its set of vertices. The quiver $Q$ extends to the repetition quiver $\widehat{Q}$, whose set of vertices $\widehat{I}$ is a subset of $I\times \mathbb{Z}$ (see section \ref{subsect_Happel}). For any $\widehat{I}$-graded vector space $W$, Nakajima constructed the graded quiver variety $\ssM_0^{\bullet}(W)$, an affine variety obtained through a geometric quotient. Moreover, he shows that any graded quiver variety admits a stratification
\begin{equation}
    \ssM_0^{\bullet}(W)=\bigsqcup_V \ssM_0^{\bullet\mathrm{reg}}(V,W).
\end{equation}
where the  \emph{regular strata} $\ssM_0^{\bullet\mathrm{reg}}(V,W)$ are parametrized by vector spaces $V$ graded over the complement of $\widehat{I}$.
Among other applications, Nakajima computed the multiplicity of the simple modules in the standard modules of $U_q(L\mathfrak{g})$ through the intersection cohomology of the above strata.
Despite their importance, the definition of graded quiver varieties and of their strata is not really explicit. Nevertheless, various results offer equivalent descriptions: 
\begin{itemize}
    \item Hernandez--Leclerc \cite{HL_quantum_Groth_rings_derived_Hall} realized certain graded quiver varieties as varieties of representations of the quiver $Q$ such that, under this correspondence, the regular strata correspond to the orbits of representations of $Q$.  
    \item Leclerc--Plamondon \cite{Leclerc_Plamondon_2013} extended Hernandez--Leclerc's result by realizing a larger class of graded quiver varieties as varieties of representation of the repetitive algebra, maintaining the bijection among strata and orbits. 
    \item Keller--Scherotzke \cite{Keller_Scherotzke_2016}, through a categorical approach, generalized these results to any graded quiver varieties associated to an acyclic quiver $Q$, as defined by Qin \cite{Qin_tqchar_2014}. 
    \item Along the same categorical direction, recently  Canesin \cite{canesin2026categoriessplitfiltrationsgraded} investigated the $n$-fold affine graded tensor product varieties.
\end{itemize}

In Keller--Scherotzke's work, different partial orders come into play:
\begin{itemize}
    \item the degeneration order on the set of strata of a quiver variety.
    \item Nakajima's order $\le$ on the set $\mathcal{M}^+$ of \emph{dominant monomials}, that is, the monic monomial in the variables $Y_{i,p}$, for $(i,p)$ in $\widehat{I}$ (see \ref{subsect_Nakajima_order}).
    \item Jensen--Su--Zimmerman's degeneration order $\le_{\Delta}$ on the set of objects of the derived category $D^b(Q)$.  
\end{itemize}
These sets are related in the following way:
 each stratum $\ssM_0^{\bullet\mathrm{reg}}(V,W)$ is naturally  associated to a monomial $m(V,W)$ in the variables $Y_{i,p}^{\pm}$. It is proved in \cite{Qin_tqchar_2014} that a stratum is non-empty if and only if the monomial is dominant and that, under this assignment, Nakajima's order on dominant monomials corresponds to the degeneration order among strata.
On the other hand, using Happel's equivalence, we obtain a bijective correspondence $m\leftrightarrow V(m)$ between the set of dominant monomials and the isoclasses of indecomposable objects of $D^b(Q)$.
Finally, for any graded quiver variety $\ssM_0^{\bullet}(W)$, Keller--Scherotzke proved that the map $\phi$ that associates to the stratum $\ssM_0^{\bullet\mathrm{reg}}(V,W)$ the  indecomposable object $V(m(V,W))$ respects the  partial orders.
To sum things up, for any $\widehat{I}$-graded vector space $W$, we have a diagram of partially ordered sets
\[\begin{tikzpicture}[
   hookright/.style={{Hooks[right]}->},
   hookleft/.style={{Hooks[left]}->},]
\node (M) at (0,0) {$\{ \ssM_0^{\bullet\mathrm{reg}}(V,W)\ |\ \ssM_0^{\bullet\mathrm{reg}}(V,W)\neq \emptyset\}$};
\node (m) at (8,0) {$\{m'\in \mathcal{M}^+\ |\ m'\leq m(0,W)\}$};
\node (D) at (4,-2.5) {$\{N\in \mathrm{Iso}(D^b(Q))\ |\ N\leq V(m(V,W))\}$};
\draw[->] (M) --  
(m) node[midway, above] {$\sim$}
node[midway, below] {$m(-,-)$};
\draw[hookright] (M) --  
(D) node[midway, below, sloped] {$\phi$} ;
\draw[hookleft] (m) --  
(D) node[midway, below, sloped] {$V(-)$};
\end{tikzpicture}
\]

The aim of this article is to complete the above picture by showing that, for Dynkin quivers, the above embeddings are actually isomorphisms of partially ordered sets. To this end, we start by presenting several results concerning the degeneration order on $D^b(Q)$. Inspired by an analogous result for the degeneration order on the module category of $Q$ (see \cite{Riedtmann1986,Bongartz1996,Zwara1998}), we show that, if $N\le_{\Delta}M$, then there exists a chain of degenerations
\[N=M_0\le_{\Delta}M_1\le_{\Delta}\cdots\le_{\Delta}M_s=M\]
such that, for any $i$, there is a decomposition $M_i\cong U_i\oplus V_i$ fitting into a triangle
\[U_i\to M_{i-1}\to V_i\to U_i[1].\]
Moreover, we define a notion of minimal degeneration (a little bit different from \cite{Zwara_2000}) and show that  each $N\le_{\Delta}M$ can be decomposed into a chain of minimal degenerations. Finally, we prove that each minimal degeneration $V(n)\le_{\Delta} V(m)$ satisfies $n\le m$, by relying on Hernandez--Leclerc's isomorphism between the quantum Grothendieck ring of $U_q(L\mathfrak{g})$ and the derived Hall algebra of $D^b(Q)$. This will lead us to the following result:
\begin{teo}[Theorem \ref{Thm:Correspondence}] For any $n,m\in \sM^+$,
    \[n\le m\qquad \text{ if and only if}\qquad V(n)\le_{\Delta}V(m).\]
\end{teo}

The article, in addition to this introduction, is composed of two sections. In section \ref{sec:Preliminanry} we collect results and terminology on dominant monomials, graded quiver varieties and derived categories. In Section \ref{sec:degenerations}, we study the degeneration order on the derived category and provide our main theorem.

\medskip
\section{Preliminaries}\label{sec:Preliminanry}
In this section, we define the partial orders on dominant monomials, isoclasses of objects of the derived category and strata of quiver varieties. Moreover, we recall a few notions about derived Hall algebras.

Let $Q$ be a Dynkin quiver. write $I=\{1,\dots,n\}$ for its set of vertices. Let $\mathbf{k}$ be a field. Denote by $\mathrm{mod}\mathbf{k}Q$ the category of finite dimensional modules of its path algebra. We write $D^b(Q)$ for its bounded derived category, with shift functor $[1]$.

\subsection{Happel's correspondence}
\label{subsect_Happel}
 We fix a \emph{height function} $\widehat{\varepsilon}$ on the quiver $Q$, that is, a function 
\[ \widehat{\varepsilon}: I\rightarrow \mathbb{Z}, i\mapsto \widehat{\varepsilon}_i\]
such that, if there is an arrow $i\rightarrow j$, then $\widehat{\varepsilon}_i-\widehat{\varepsilon}_j=1$. Let $\widehat{I}$ be the index set
\[ \widehat{I}=\{(i,p)\in I\times\mathbb{Z}\ |\  | p-\widehat{\varepsilon}_i| \in 2\mathbb{Z} \}.\]

The repetition quiver $\widehat{Q}$ is the quiver whose vertex set is $\widehat{I}$ and with arrows \[(i,p) \rightarrow (j,p+1),\] for any $(i,p) \in \widehat{I}$ and adjacent vertices $i$ and $j$ in $Q$.
The \emph{mesh category} $\mathbf{k}(\widehat{Q})$  is the $\mathbf{k}$-linear category whose objects are identified with the vertices of $\widehat{Q}$ and
whose spaces of morphisms $\mathrm{Hom}_{\mathbf{k}(\widehat{Q})}(A,B)$ are the spaces of linear combinations of paths from $A$ to $B$, modulo the \emph{mesh relations}: for each $(i,p)$ in $\widehat{I}$, the sum of all paths from $(i,p)$ to $(i,p+2)$ vanishes. The composition of morphisms is induced by the composition of paths.

We recall a fundamental result due to Happel. Let $I_i$ be the indecomposable injective associated to the vertex $i$ and write $\tau$ for the Auslander--Reiten translation of $D^b(Q)$. We write $\mathrm{Ind}D^b(Q)$ for the full subcategory of $D^b(Q)$ formed by the indecomposable objects.
\begin{teo}[{\cite[Chapter 1, Proposition 5.6]{Happel_articolo_1987}}]\label{Happel's Thm}
There is an equivalence of $\mathbf{k}$-categories
\[V(-):\mathbf{k}(\widehat{Q})\xrightarrow{\sim} \mathrm{Ind}D^b(Q)\]
such that, for any $(i,p)$ in $\widehat{I}$, we have
\[V(i,p)=\tau^{(\widehat{\varepsilon}_i-p)/2}(I_i).\]
\end{teo}

\subsection{Degeneration order in $D^b(Q)$}

\begin{de}[{\cite{Jensen_Su_Zimmerman_2005II_tring_cat, Jense_Su_Zimmermann_2005_der_cat}}]
    Let $X$ and $Y$ be in $D^b(Q)$. The object $X$ \emph{degenerates} to $Y$ (or, alternatively, $Y$ is a \emph{degeneration} of $X$) if there exists an object $Z$ of $D^b(Q)$ fitting into a triangle
    \[ Y\rightarrow X\oplus Z\rightarrow Z\rightarrow Y[1].\]
    In this case, we write \[X\leq_{\Delta} Y.\]
\end{de}
 Note that we choose the opposite direction for the inequality with respect to \cite{Jensen_Su_Zimmerman_2005II_tring_cat, Jense_Su_Zimmermann_2005_der_cat}. As shown in \cite{Jensen_Su_Zimmerman_2005II_tring_cat}, the relation $\leq_{\Delta}$ is a partial order on the set of isomorphism classes of objects of the derived category $D^b(Q)$.

\begin{de}
    Let $X$ and $Y$ be in $D^b(Q)$. The object $X$ \emph{ext-degenerates} to $Y$ (or, alternatively, $Y$ is a \emph{ext-degeneration} of $X$) if $Y$ decomposes as $Y\cong Y_1\oplus Y_2$ and there exists  a triangle
    \[ Y_1\rightarrow X\rightarrow Y_2\rightarrow Y_1[1].\]
    In this case, we write \[X\leq_{\mathrm{ext}} Y.\]
\end{de}
Notice that if $Y$ ext-degenerates to $X$, then it also degenerates to $X$,  as we have the triangle  $X_1\oplus X_2\to Y\oplus X_2\to X_2 \to (X_1\oplus X_2)[1]$.

\subsection{Quantum Grothendieck rings}
\label{subsect_Nakajima_order}
The set $\widehat{I}$ appears as an index set also in the representation theory of quantum affine algebras. For any $(i,p)$ in $\widehat{I}$, consider the variable $Y_{i,p}$.  Let $\mathcal{M}$ be the set of monic monomials  in the variables $(Y_{i,p}^{\pm})_{(i,p)\in \widehat{I}}$. We say that a monomial $\mathcal{M}$ is \textit{dominant} if it contains only positive powers of the $Y_{i,p}$. Let $\mathcal{M}^+$ be the set of dominant monomials.
Nakajima defined a partial order $\le$ on $\sM$: for any $(i,p)$ in $\widehat{I}$, set
\[ A_{i,p+1}=Y_{i,p}Y_{i,p}\prod_{i\sim j}Y_{j,p+1}^{-1}.\]
For $m,m'$ in $\mathcal{M}$, we say  $m'\leq m$ if $m'$ can be written in the form
\[ m' = m A^{-1},\]
where $A$ is a finite product of certain $A_{i,p+1}$, for $(i,p)\in \widehat{I}$.

Let $\mathfrak{g}$ be a complex simple Lie algebra of same Dynkin type as $Q$ and let $U_q(L\mathfrak{g})$ be the associated quantum loop algebra over  $\mathbb{C}$. It follows from the work of Frenkel-Reshetikhin \cite{FrenkelReshetikhin98} that the dominant monomials in $\mathcal{M}^+$ parametrize the simple objects of a full Serre subcategory, introduced by Hernandez--Leclerc \cite{HL_quantum_Groth_rings_derived_Hall} and denoted $\mathscr{C}_\mathfrak{g}^\mathbb{Z}$, of the monoidal category $\mathrm{mod}U_q(L\mathfrak{g})$ of finite-dimensional $U_q(L\mathfrak{g})$-modules. Understanding the category $\mathrm{mod}U_q(L\mathfrak{g})$ essentially reduces to the study of the category $\mathscr{C}_\mathfrak{g}^\mathbb{Z}$.

Even if the monoidal category $\mathrm{mod}U_q(L\mathfrak{g})$ is not symmetric, the Grothendieck ring  $\mathcal{K}(\mathscr{C}_\mathfrak{g})$ is commutative, as proved by Frenkel-Reshetikhin \cite{FrenkelReshetikhin98} through  the construction of the \emph{$q$-character} morphism, which is an injective ring homomorphism \[\chi_q: \mathcal{K}(\mathscr{C}_\mathfrak{g}^\mathbb{Z})\rightarrow \mathbb{Z}[Y_{i,p}^\pm]_{(i,p)\in \widehat{I}}.\]
Following \cite{Hernandez2004_alg_approach}, the quantum Grothendieck ring $\mathcal{K}_t(\mathscr{C}_\mathfrak{g}^\mathbb{Z})$ is a non-commutative version of the Grothendieck ring of $\mathscr{C}_\mathfrak{g}^\mathbb{Z}$, contained in a  \emph{quantum torus} over  $\Z[t^{\pm 1/2}]$ in the variables $(\leftidx{^t}{Y}_{i,p}^{\pm 1})_{(i,p)\in \widehat{I}}$. The ring $\mathcal{K}_t(\mathscr{C}_\mathfrak{g}^\mathbb{Z})$ is endowed with the  \emph{bar involution} $\overline{(-)}$, which is the $\mathbb{Z}$-algebra anti-involution defined by
    \[ \overline{t^{1/2}}=t^{-1/2},\ \ \overline{\leftidx{^t}{Y}_{i,p}}=\leftidx{^t}{Y}_{i,p}.     \]
The dominant monomials $m$ parametrizes two bases of $\mathcal{K}_t(\mathscr{C}_\mathfrak{g}^{\mathbb{Z}})$: the \emph{$(q,t)$-characters of the simple modules }$(L_t(m))_m$ and the \emph{$(q,t)$-characters of the standard modules }$(M_t(m))_m$ (see \cite{Hernandez2004_alg_approach}).  For any dominant monomial $m=\prod_{k=1}^r Y_{i_k,p_k}$ with $p_1\le \cdots \leq p_r$, the $(q,t)$-character $M_t(m)$ is defined in \cite{HL_quantum_Groth_rings_derived_Hall} as a product
\begin{equation}\label{qtStd}
M_t(m)=t^{\alpha(m)} L_t(Y_{i_r,p_r})\cdots L_t(Y_{i_2,p_2})L_t(Y_{i_1,p_1})    
\end{equation}
for some power 
$\alpha(m)\in \Z[\frac{1}{2}]$, and, since the $(q,t)$-characters of the simple modules are invariant under the bar involution, $\overline{M_t(m)}$ is given by
\begin{equation}\label{eq:barQTStd}
\overline{M_t(m)}=t^{-\alpha(m)} L_t(Y_{i_1,p_1}) L_t(Y_{i_2,p_2})\cdots L_t(Y_{i_r,p_r}).
\end{equation}
Moreover, by \cite{Nakajima2004}, the bar-involution of $M_t(m)$ satisfies the following equation:
\begin{equation}\label{barInv}
    \overline{M_t(m)}=M_t(m)+\sum_{m'\in \sM^+,\ m'<m} U_{m,m'}M_t(m'),
\end{equation}
for some coefficients $U_{m,m'}$ in $\mathbb{Z}[t^{\pm 1/2}]$.

\subsection{Derived Hall algebras}
\label{subsection_derived_Hall_algebra}
In the following, we fix a prime number $q$ and  set $v=q^{\frac{1}{2}}$. Let $\mathbf{k}=\F_q$ be the finite field of cardinality $q$. 

Denote the set of isomorphism classes of objects in $D^b(Q)$ by $\mathrm{Iso}D^b(Q)$. Following \cite{Toen2006, Xiao-Xu2008, Sheng2010}, the \emph{derived Hall algebra} of $D^b(Q)$, denoted by $\sD\sH(Q)$, is the associative $\mathbb{Q}(v^{1/2})$-algebra whose underlying vector space has a basis indexed by $\mathrm{Iso}D^b(Q)$, and with multiplication given by 
$$u_Xu_Y =\sum_{Z\in \mathrm{Iso}D^b(Q)} F_{X,Y}^Z u_Z,$$
where $F_{X,Y}^Z$ is a non-negative coefficient, different from zero if and only if there is a triangle $X\to Z\to Y\to X[1]$. 
In \cite{HL_quantum_Groth_rings_derived_Hall} , Hernandez--Leclerc constructed an isomorphism 
\[ \Phi: \mathbb{Q}(v^{1/2})\otimes_{\mathbb{Z}[t^{\pm 1/2}]}\mathcal{K}_t(\mathscr{C}_\mathfrak{g}^{\mathbb{Z}}) \cong \sD\sH(Q),\qquad L_t(Y_{i,p})\mapsto v^{\frac{1}{2}}(v-v^{-1})u_{V(i,p)},\]
where the $\mathbb{Z}[t^{\pm 1/2}]$-module structure on $\mathbb{Q}(v^{1/2})$ is induced by the specialization $t^{1/2}\mapsto v^{1/2}$.
Notice that, under this isomorphism, the basis of the $(q,t)$-characters of the standard modules is mapped to the a rescaling of the basis $\{u_X|~X\in \mathrm{Iso}D^b(Q)\}$ of $\sD\sH(Q)$.

\subsection{Graded quiver varieties}
Following \cite{Keller_Scherotzke_2016,Nakajima2001}, we end this section by recalling some properties of graded quiver varieties. By \emph{vector space} we mean a finite-dimensional vector space over the field of complex numbers. Let $\widehat{I}'$ be the index set complementary to $\widehat{I}$ in $I\times\mathbb{Z}$.
For any  $\widehat{I}$-graded vector space $W=\bigoplus_{(i,p)\in \widehat{I}}W_{i,p}$, let $\mathfrak{M}^\bullet_0(W)$ be the associated graded quiver variety. It is an affine complex algebraic variety endowed with a natural stratification, whose strata, denoted $\mathfrak{M}_0^{\bullet\mathrm{reg}}(V,W)$, are indexed by $\widehat{I}'$-graded vector spaces $V$. For any $\widehat{I}'$-graded vector space $V$ and $\widehat{I}$-graded vector space $W$, let $m(V,W)$ be the dominant monomial defined by
\[
m(V,W)=\prod_{(j,s)\in \widehat{I}} Y^{\mathrm{dim}(W_{j,s})} \prod_{(i,p)\in \widehat{I}'} A_{i,p}^{-\mathrm{dim}(V_{i,p})}.
\]
Notice that $m(V,W)\leq m(0,W)$ with respect to Nakajima's order.
The following result describes the correspondence between Nakajima's order on dominant monomials and the degeneration order on strata:

\begin{prop}[{\cite[Prop. 4.3.13, Cor. 4.3.15]{Qin_tqchar_2014}}]
\label{prop_qin}
Let $V$ and $V'$ be $\widehat{I}'$-graded vector spaces and let $W$ be a $\widehat{I}$-graded vector space.
\begin{enumerate}
    \item the stratum $\mathfrak{M}_0^{\bullet\mathrm{reg}}(V,W)$ of $\mathfrak{M}_0^{\bullet}(W)$ is not-empty if and only if the monomial $m(V,W)$ is dominant.
    \item $\mathfrak{M}_0^{\bullet\mathrm{reg}}(V',W)\subset \overline{\mathfrak{M}_0^{\bullet\mathrm{reg}}(V,W)}$ if and only if
  \[ m(V,W)\leq m(V',W).\]
\end{enumerate}

\end{prop}


In \cite{Keller_Scherotzke_2016}, Keller--Scherotzke construct a categorical version $\mathcal{M}^\bullet_0(W)$ of Nakajima's graded quiver variety. They define a canonical functor $\mathcal{M}^\bullet_0(W)\rightarrow D^b(Q)$ that induces a map $\phi$ from $ \mathfrak{M}^\bullet_0(W)$ to the set of isoclasses of $D^b(Q)$.

\begin{lem}[{\cite[Lemma 4.14]{Keller_Scherotzke_2016}}]
\label{lem_KellerScheroktze_map}
For any $x$  in $\mathfrak{M}_0^{\bullet\mathrm{reg}}(V,W)$, we have
\[  \phi(x)\cong V(m(V,W)).  \]

\end{lem}

Moreover, Keller-Scherotzke show that, under the map $\phi$, the degeneration of strata corresponds to the degeneration order in the derived category.

\begin{prop}[{\cite[Theorem 2.8]{Keller_Scherotzke_2016}}]
\label{pro_KelScher_bijection}
Let $W$ be a $\widehat{I}$-graded vector space and let $V$ and $V'$ be  $\widehat{I}'$-graded vector spaces such that the regular strata $\mathfrak{M}_0^{\bullet\mathrm{reg}}(V,W)$ and  $\mathfrak{M}_0^{\bullet\mathrm{reg}}(V',W)$ of $\mathfrak{M}_0^{\bullet}(W)$ are not empty. Let $m=m(V,W)$ and $m'=m(V',W)$. Then \[ \mathfrak{M}_0^{\bullet\mathrm{reg}}(V',W)\subset \overline{\mathfrak{M}_0^{\bullet\mathrm{reg}}(V,W)}\quad \text{if and only if}\quad V(m)\leq_{\Delta} V(m')\].
\end{prop}
Hence, combining the above result with Proposition \ref{prop_qin}, for any dominant monomial $m$, we have an embedding of partially ordered sets
 \[V(-):\{n\in \sM^+\ |\ n\le m\}\hookrightarrow \{N\in \mathrm{Iso}D^b(Q)\ |\ N\le_{\Delta}V(m)\}.\]
In the next section, we will show that this is actually an isomorphism.

\section{ Correspondence between partial orders}\label{sec:degenerations}
Recall that $Q$ is a Dynkin quiver.
We start this section by studying the degeneration order in the derived category $D^b(Q)$.

\subsection{Chains of ext-degenerations}

 For any $X\le_{\Delta}Y$ in $D^b(Q)$ with a triangle 
 $$Y\to X\oplus Z\to Z\to Y[1],$$
 we first consider the case where $Z$ is indecomposable. The following lemma states that, in this case, $X\le_{\Delta}Y$ is induced by an ext-degeneration.
 
\begin{lem}
\label{lem_degeneration_with_Z_indec} 
  For any $X,Y\in D^b(A)$, if there exists an indecomposable object $Z$ of $D^b(A)$ such that there  exists a triangle
    \[ Y \xrightarrow{f} X\oplus Z\xrightarrow{g} Z\rightarrow Y[1].\]
 then $X\cong Y$ or there exists an object $Y'$ in $D^b(A)$ and a triangle 
        \[ Y'\rightarrow X \rightarrow Z\rightarrow Y'[1],\]
    such that $Y\cong Y'\oplus Z$.
\end{lem}
\begin{proof}
    Write $g=(g_X,g_Z)$ for the components of $g$. Since $Z$ in indecomposable, we have $g_Z\in \mathrm{End}_{D^b(A)}\cong \mathbf{k}$.\\
    If $g_Z=0$, then $\mathrm{cone}(g)\cong Y[1]$ is isomorphic to the direct sum $\mathrm{cone}(g_X)\oplus Z[1]$. Therefore, we obtain the statement by setting $Y'=\mathrm{cone}(g_X)[-1]$.\\
    If $g_Z\neq 0$, the it is a scalar multiple of $Id_Z$. Therefore, $g$ is a retraction. If follows directly that $Y\oplus Z$ is isomorphic to the direct sum $X\oplus Z$ and, therefore, $X\cong Y$.
\end{proof}

For general $Z$, to show that $X\le_{\Delta}Y$ can be decomposed into a chain of ext-degenerations, we need the following lemma. It is a simple application of the octahedral axiom.
\begin{lem}
\label{lem_tring_cat_axioms}
    For any triangle $Y\rightarrow X\rightarrow Z_1\oplus Z_2\xrightarrow{(h_1,h_2)} Y[1]$ in $D^b(Q)$, there exist:
    \begin{enumerate}
        \item triangles
              $Y\rightarrow X_i \rightarrow Z_i\xrightarrow{h_i} Y[1]$, for $i=1,2$;
        \item triangles $X_i\rightarrow X \rightarrow Z_{\overline{i}} \rightarrow X_i[1]$, for $i=1,2$ and $\overline{i}\in \{1,2\}\backslash \{i\}$.
        \end{enumerate}
\end{lem}
\begin{proof}
    The triangles in $(1)$ are constructed by setting $X_i=\mathrm{cone}(h_i)[-1]$. 
    For the triangles in $(2)$,
    , let $\iota_i$ be the inclusion $Z_i\rightarrow Z_1\oplus Z_2$ and apply the octahedron axiom to the composition $h_i=h\circ \iota_i$. Therefore, we obtain a triangle
     \[ \mathrm{cone}(\iota_i)\rightarrow \mathrm{cone}(h'_i)\rightarrow \mathrm{cone}(h)\rightarrow \mathrm{cone}(\iota_i)[1],\]
     which is isomorphic to the triangle
     \[ Z_{\overline{i}}\rightarrow X_i[1]\xrightarrow{t} X[1]\rightarrow Z_{\overline{i}}[1],\]
     which, up to taking shifts, is the desired triangle.
\end{proof}

Now we can generalize Lemma \ref{lem_degeneration_with_Z_indec} to any degeneration $X\leq_\Delta Y$. 

\begin{prop}\label{prop:DegOrder}
 Let $X$ and $Y$ be objects in $D^b(Q)$ such that $X\leq_{\Delta} Y$. Then there exists a positive integer $s$ and a sequence of ext-degenerations 
 \[ X\cong M_1 \leq_{\mathrm{ext}} M_2\leq_{\mathrm{ext}}\dots \leq_{\mathrm{ext}} M_{s-1} \leq_{\mathrm{ext}} M_s \cong Y.\]
\end{prop}

\begin{proof}
    Since $X\leq_{\Delta} Y$, there exists an object $Z$ in $D^b(Q)$ and a triangle
     \begin{equation}    \label{eq_triangle_Z_proof_proposition_chain_degenerations}
     Y\rightarrow X\oplus Z\rightarrow Z\xrightarrow{h} Y[1].
     \end{equation} 
     Write $Z=\bigoplus_{1\leq i\leq r} Z_i^{a_i}$ for the direct sum decomposition of $Z$ into indecomposable summands. Up to reordering, we can assume that $\mathrm{Hom}_{D^b(Q)}(Z_i,Z_j)=0$ whenever $i>j$. We proceed by induction on $r$. When $r=1$, the triangle (\ref{eq_triangle_Z_proof_proposition_chain_degenerations}) is of the form
     \begin{equation}
     \label{eq_triangle_Z1_proof_proposition_chain_degenerations}
         Y\rightarrow X\oplus Z_1^{a_1}\rightarrow Z_1^{a_1}\xrightarrow{h} Y[1].
     \end{equation}
     We proceed by induction on $a_1$. If $a_1=1$, the statement is a reformulation of Lemma \ref{lem_degeneration_with_Z_indec}. 
     Suppose that $a_1\geq 2$. Then by applying Lemma \ref{lem_tring_cat_axioms} to the triangle (\ref{eq_triangle_Z1_proof_proposition_chain_degenerations}), we obtain triangles
     \[ Y\rightarrow \overline{X} \rightarrow Z_1^{a_1-1}\rightarrow Y[1] \quad \text{ and } \quad 
     \overline{X}\rightarrow X\oplus Z_1^{a_1} \rightarrow Z_1\rightarrow \overline{X}[1].\]
     By applying Lemma \ref{lem_degeneration_with_Z_indec} to the latter triangle, we have $\overline{X}\cong X\oplus Z_1^{a_1-1}$ or there exist an object $X_1$ in $D^b(Q)$ such that $\overline{X}\cong X_1\oplus Z_1$ and a triangle 
     \begin{equation}
     \label{eq_triangle_X1_proof_proposition_chain_degenerations}
         X_1\rightarrow X\oplus Z_1^{a_1-1}\rightarrow Z_1\rightarrow X_1[1].
         \end{equation}
     If $\overline{X}\cong X\oplus Z_1^{a_1-1}$, then we have a triangle $Y\rightarrow X\oplus Z_1^{a_1-1} \rightarrow Z_1^{a_1-1}\rightarrow Y$, to which we can apply the induction hypothesis. 
     Otherwise, we can apply Lemma \ref{lem_degeneration_with_Z_indec} to the triangle (\ref{eq_triangle_X1_proof_proposition_chain_degenerations}): if $X_1\cong X\oplus Z_1^{a_1-2}$ then we are again in the case $\overline{X}\cong X\oplus Z_1^{a_1-1}$; otherwise, there exist an object $X_2$ in $D^b(Q)$ such that $X_1\cong X_2\oplus Z_1$ and a triangle 
     $ X_2\rightarrow X\oplus Z_1^{a_1-2}\rightarrow Z_1\rightarrow X_2[1].
    $ By iterating this process, if $\overline{X}\ncong X\oplus Z_1^{a_1-1}$, we end up with a sequence of objects $(X_i)_{1\leq i\leq a_1}$ such that
    $\overline{X}\cong X_1\oplus Z_1\cong X_2\oplus Z_1^2\cong \dots \cong X_{a_1}\oplus Z_1^{a_1}$ and with a triangle
    \begin{equation}
    \label{eq_triangle_Xa_1+1_proof_proposition_chain_degenerations}
        X_{a_1}\rightarrow X\rightarrow Z_1\rightarrow X_{a_1}[1].
    \end{equation}
    By applying the induction hypothesis to the triangle $Y\rightarrow X_{a_1}\oplus Z^{a_1}_1\rightarrow Z_1^{a_1-1}\rightarrow Y[1]$, we obtain a sequence of objects $(M_i)_{1\leq i\leq s}$ such that $M_1=X_{a_1}\oplus Z_1$, $M_{s}\cong Y$ and, for any $i=2,\dots, s$, $M_i$ decomposes as a direct sum $M_i\cong U_i\oplus V_i$ such that there exists a triangle
    \[ U_i\rightarrow M_{i-1} \rightarrow V_i \rightarrow U_i[1].\]
    Therefore, thanks to the triangle (\ref{eq_triangle_Xa_1+1_proof_proposition_chain_degenerations}), by completing the sequence with $M_0=X$, we obtain the desired statement.

     Suppose now $r\geq 2$ and write $Z\cong Z'\oplus Z_1^{a_1}$. By applying Lemma \ref{lem_tring_cat_axioms}, we obtain triangles
     \begin{align} Y\rightarrow Y' \rightarrow Z'\rightarrow Y[1] \quad \text{ and } \label{eq_triangle_Y_Y'_proof_proposition_chain_degenerations}\\ 
     Y'\rightarrow X\oplus Z'\oplus Z_1^{a_1} \rightarrow Z_1^{a_1}\rightarrow Y'[1].\label{eq_triangle_Y'Z_1a_1_proof_proposition_chain_degenerations}
     \end{align}
     We proceed by induction on $a_1$. Assume first that $a_1=1$. By Lemma \ref{lem_degeneration_with_Z_indec}, $Y'\cong X\oplus Z'$ (and in this case, via the triangle (\ref{eq_triangle_Y_Y'_proof_proposition_chain_degenerations}), we can conclude by the induction hypothesis on $r$) or there exists an object $X_{1}'$ such that $Y'\cong X_{1}'\oplus Z_1$ and a short exact sequence  
      \begin{equation} 
     X'_{1}\rightarrow X\oplus Z'  \xrightarrow[]{(g_X,g_{Z'})} Z_1\rightarrow X'_{1}[1].
     \end{equation}
     In the latter case, note that, by the chosen ordering of the direct summands of $Z$, $g_{Z'}=0$. Therefore, 
     \[ X'_{1}\cong \mathrm{cone}(g_X)[-1]\oplus Z'.\] Therefore, by setting $\widetilde{X}=\mathrm{cone}(g_{X})[-1]$, we have a triangle
     \begin{equation}
         \widetilde{X}\rightarrow X\xrightarrow[]{g_x} Z_1\rightarrow \widetilde{X}[1].
         \label{eq_triangle_Xtilde_X_proof_proposition_chain_degenerations}
     \end{equation}
     Moreover, the triangle (\ref{eq_triangle_Y_Y'_proof_proposition_chain_degenerations}) is isomorphic to a triangle 
     \begin{equation}
         Y\rightarrow \widetilde{X}_1\oplus Z_1\oplus Z' \rightarrow Z'\rightarrow Y[1] .
     \end{equation}
     Applying the induction hypothesis on $r$ to the last triangle, we obtain a sequence $(M'_{i})_{1\leq i\leq s}$ with $M_1'=\widetilde{X}\oplus Z_1$ and $M_s'=Y$. Finally, via the triangle $(\ref{eq_triangle_Xtilde_X_proof_proposition_chain_degenerations})$, we can complete the sequence with $M'_0=X$.

    Assume now that $a_1\geq 2$.
    We proceed with the triangle (\ref{eq_triangle_Y'Z_1a_1_proof_proposition_chain_degenerations}) as we have done for triangle (\ref{eq_triangle_Z1_proof_proposition_chain_degenerations}). In particular, there are two cases:\\
    Case 1: there exists a triangle
    \begin{equation}
        Y'\rightarrow X\oplus Z'\oplus Z_1^{a_1-1} \rightarrow Z_1^{a_1-1}\rightarrow Y'[1].
    \end{equation}
    In this case, we can apply the inductive hypothesis on $a_1$.

    Case 2: we get two triangles 
      \begin{align} Y'\rightarrow X'_{a_1}\oplus Z_1^{a_1} \rightarrow Z_1^{a_1-1}\rightarrow Y'[1] \quad \text{ and }
      \label{eq_triangle_Y'X'_a1_proof_proposition_chain_degenerations}\\ 
     X'_{a_1}\rightarrow X\oplus Z'  \xrightarrow[]{(g_X,g_{Z'})} Z_1\rightarrow X'_{a_1}[1].
     \end{align}
     As before, we have $X'_{a_1}\cong \tilde{X}\oplus Z'$ and a triangle
     \begin{equation}
         \widetilde{X}\rightarrow X\xrightarrow[]{g_x} Z_1\rightarrow \widetilde{X}[1].
         \label{eq_triangle_Xtilde_XX_proof_proposition_chain_degenerations}
     \end{equation}
     By a similar reasoning applied to the triangle (\ref{eq_triangle_Y'X'_a1_proof_proposition_chain_degenerations}), there exist an object $Y''$ such that $Y'\cong Y''\oplus Z'$ and a triangle
     \begin{equation}
         Y''\rightarrow \widetilde{X}\oplus Z_1^{a_1}\rightarrow Z_1^{a_1-1}\rightarrow Y''[1].
     \end{equation}
     In particular, the triangle (\ref{eq_triangle_Y_Y'_proof_proposition_chain_degenerations}) is isomorphic to a triangle 
     \begin{equation}
         Y\rightarrow Y''\oplus Z' \rightarrow Z'\rightarrow Y[1] .
     \end{equation}
     Applying the induction hypothesis to these last two triangles, we obtain a sequence $(M'_i)_{1\leq i\leq s}$ with $M_1'=\widetilde{X}\oplus Z_1$ and $M_s'=Y$. Finally, via the triangle $(\ref{eq_triangle_Xtilde_XX_proof_proposition_chain_degenerations})$, we can complete the sequence with $M'_0=X$.
\end{proof}

 For any $Y\cong\bigoplus_{k=1}^r V(i_k,p_k)\in D^b(Q)$, define
\[p_{min}(Y)=\mathrm{min}\{p_k|~1\le k\le r\},\qquad p_{max}(Y)=\mathrm{max}\{p_k|~1\le k\le r\}.\]

\begin{lem}\label{lem:Boundary}
    Let $X$ and $Y$ be objects of $D^b(Q)$. If $X\leq_\mathrm{\Delta} Y$, then \[p_\mathrm{min}(X)\geq p_\mathrm{min}(Y) \quad  \text{and} \qquad p_\mathrm{max}(X)\leq p_\mathrm{max}(Y).\]
\end{lem}
\begin{proof}
We prove that $\qquad p_\mathrm{max}(X)\leq p_\mathrm{max}(Y)$, the other case being analogous. By proposition \ref{prop:DegOrder}, we can assume $X\le_{\mathrm{ext}}Y$. Therefore, there exists a decomposition $Y\cong Y'\oplus Y''$ with a triangle
 \[Y'\to X\to Y''\to Y'[1].\]
 We must have $p_{max}(X)\le p_{max}(Y)$. Otherwise, there exists a direct summand $X_k=V(i_k,p_k)$ of $X$ such that $p_k>p_{max}(Y)$, which implies $\Hom_{D^b(Q)}(X_k,Y'')=0$. Hence $X_k$ would be an indecomposable direct summand of $Y'$ and $p_{max}(Y')\ge p_k$, which leads to a contradiction.
\end{proof}

Set 
$$Y^{max}=\bigoplus_{k,\ p_k=p_{max}(Y)} V(i_k,p_k),\qquad Y^{-}=\bigoplus_{k,\ p_k<p_{max}(Y)}V(i_k,p_k),$$
and 
$$Y^{min}=\bigoplus_{k,\ p_k=p_{min}(Y)} V(i_k,p_k),\qquad Y^{+}=\bigoplus_{k,\ p_k>p_{min}(Y)}V(i_k,p_k),$$
Then $Y\cong Y^-\oplus Y^{max}\cong Y^{min}\oplus Y^+$.

For any $X\le_{\Delta} Y$, with  $X\cong \bigoplus_{k}V(i'_k,p'_k)$, we set
  $$X^{Y,max}=\bigoplus_{k,\ p'_k= p_{max}(Y)}V(i'_k,p'_k),\qquad X^{Y,-}=\bigoplus_{k,\ p'_k< p_{max}(Y)}V(i'_k,p'_k).$$
Then we have $X\cong X^{Y,-}\oplus X^{Y,max}$.

For any object $X$ of $D^b(Q)$, let $|X|$ be the set of its indecomposable direct summands, counted with their multiplicities. We write $\#|X|$ for its cardinality.

\begin{lem} \label{lem:ExtDeg1}
Let $X$ and $Y$ be objects of $D^b(Q)$.  Then $X\le_{ext}Y$ if and only if there exists a decomposition $Y^{max}\cong X^{Y,max}\oplus Z$ such that $X^{Y,-}\le_{ext}Y^-\oplus Z$.
\begin{proof}
  It follows from the definition $\le_{\mathrm{ext}}$ that there exists a decomposition $Y\cong Y'\oplus Y''$ fitting into a triangle
  \[Y'\stackrel{f}\to X\stackrel{g}\to Y''\to Y'[1].\]
 If $p_{max}(X)<p_{max}(Y)$, then $X^{Y,max}=0$ and 
 \[X^{Y,-}\cong X\le_{ext}Y\cong Y^{-}\oplus Y^{max}\]
 as desired.
 
  If $p_{max}(X)=p_{max}(Y)$, then $X^{max}=X^{Y,max}$ and $X^-=X^{Y,-}$. Let $X^{max}=\bigoplus_{i=1}^s X^{max}_i$ be  the decomposition into indecomposable objects and let $\iota_s:X^{max}_s\to X$ be the canonical embedding. Write $X^{max}=X^{max}_{<s}\oplus X^{max}_s$ and 
   $$g=(g',g_s):\bigl(X^{-}\oplus X^{max}_{<s}\bigr)\oplus X^{max}_s\to Y''.$$  
  If $g_s=0$, then $Y'\cong \mathrm{Cong}(g)[-1]\cong  \mathrm{Cone}(g')[-1]\oplus X^{max}_s$ and 
   $$Y\cong Y''\oplus \mathrm{Cone}(g')[-1]\oplus X^{max}_s.$$
  It follows from $Y\cong Y^-\oplus Y^{max}$ and $p_{max}(Y^{+})<p_{max}(Y)$ that there exists $Z$ such that $Y^{max}\cong X^{max}_s\oplus Z $ and $Y^-\oplus Z\cong Y''\oplus \mathrm{Cone}(g')[-1]$. In particular, there is a  triangle
    $$\mathrm{Cone}(g')[-1]\to X^{-}\oplus X^{max}_{<s}\stackrel{g'}\to Y''\to \mathrm{Cone}(g'),$$
    which implies 
  \begin{equation}\label{Ineq:extDeg1}
        X^{-}\oplus X^{max}_{<s}\le_{ext} Y^-\oplus Z.
  \end{equation}
    
   If $g_s\ne 0$, notice that $\Hom(X^{max}_s,Y^-)=0$ and  $p_{max}(X)=p_{max}(Y)$. Moreover, we have $X^{max}_s\in |Y^{max}|$ and
    $$g_s=(0,g'_s):X^{max}_s\to \bar{Y}''\oplus (X^{max}_s)^{\oplus a}\cong Y'' .$$
  Since $g'_s\ne 0$ and $X^{max}_s$ is indecomposable, $g'_s$ is a retraction, which implies 
    $$\mathrm{Cone}(g_s)\cong \bar{Y}''\oplus \mathrm{Cone}(g'_s)\cong \bar{Y}''\oplus (X^{max}_s)^{a-1}.$$
  Moreover,  we obtain a commutative diagram of triangles from $g_s=g\circ\iota_s$:
    $$\begin{tikzcd}
      &X^{max}_s\arrow[r,equal] \arrow[d,"\iota_s"] &X^{max}_s\arrow[d,"g_s"]\\
      Y'\arrow[d,equal]\arrow[r,"f"] &X\arrow[d]\arrow[r,"g"] &Y''\arrow[r]\arrow[d] &Y'[1]\arrow[d,equal]\\
      Y'\arrow[r] &X^{-}\oplus X^{max}_{<s}\arrow[d]\arrow[r] &\mathrm{Cone}(g_s)\arrow[d]\arrow[r] &Y'[1]\\
       &X_t[1]\arrow[r,equal] &X_t[1]
    \end{tikzcd},$$
    which implies 
    \begin{equation}\label{Ineq:extDeg2}
        X^{-}\oplus X^{max}_{<s}\le_{ext} Y'\oplus \bar{Y}''\oplus (X^{max}_s)^{a-1}.
    \end{equation}
   Set $Z=\bar{Y}''\oplus (X^{max}_s)^{a-1}$.
    If $\#|X^{max}|=1$, then $X^{max}_{<s}=0$ and (\ref{Ineq:extDeg1}) and (\ref{Ineq:extDeg2}) reduce to
     $$X^{-}\le_{ext} Y^-\oplus Z,$$
    with $Y^{max}\cong Z\oplus X^{max}$. For $\#|X^{max}|\ge 2$, by applying the induction hypothesis to (\ref{Ineq:extDeg1}) and (\ref{Ineq:extDeg2}), we obtain a decomposition $Z\cong Z'\oplus X^{max}_{<s}$ such that 
    $$X^{-}\le_{ext} Y^-\oplus Z',$$
    as $\bigl(X^{-}\oplus X^{max}_{<s}\bigr)^-=X^{-}$ and $\bigl(Y^-\oplus Z\bigr)^-=Y^-$. Hence, we complete the proof by noting that $Y^{\max}\cong Z\oplus X^{max}_s\cong  Z'\oplus X^{max}_{<s}\oplus X_s^{max}$. 
\end{proof}
\end{lem}
  Analogously to the above setting, we can also define $X^{Y,min}$ and $X^{Y,min}$:
  \[X^{Y,min}=\bigoplus_{k,\ p'_k\le p_{min}(Y)} V(i'_k,p'_k),\qquad X^{Y,+}=\bigoplus_{k,\ p_k>p_{min}(Y)}V(i'_k,p'_k).\]
  Then $X\cong X^{Y,+}\oplus X^{Y,min}$ and similarly as above, we can prove:
\begin{lem}\label{lem:ExtDeg2}
For any $X,Y\in D^b(Q)$,  $X\le_{ext}Y$ if and only if there exists a decomposition $Y^{min}=X^{Y,min}\oplus Z'$ such that $X^{Y,+}\le_{ext}Y^{+}\oplus Z'$.
\end{lem}

  Recall that, by Proposition \ref{prop:DegOrder}, for any $X\le_{\Delta}Y$, we can find a chain  of ext-degenerations:
  \[X=M_0 \le_{\mathrm{ext}}M_1\le_{\mathrm{ext}}\cdots\le_{\mathrm{ext}}M_s=Y.\]
  Applying the previous two lemmas to each intermediate degeneration $M_{i-1}\le_{ext}M_i$, we obtain the following.

\begin{cor}\label{cor:ExtDeg}
 For any $X,Y\in D^b(Q)$, ~
 \begin{itemize}
     \item[(i)] $X\le_{\Delta}Y$ if and only if there exists an object $Z$ such that  $Y^{max}\cong X^{Y,max}\oplus Z$ and $X^{Y,-}\le_{\Delta}Y^-\oplus Z$.
     \item[(ii)] $X\le_{\Delta}Y$ if and only if there exists  an object $Z$ such that $Y^{min}\cong X^{Y,min}\oplus Z'$ and $X^{Y,+}\le_{\Delta}Y^{+}\oplus Z'$.
 \end{itemize}
\end{cor}

\medskip
\subsection{Minimal degenerations} \label{sec:MinDeg}
For any $X,Y\in D^b(Q)$ such that $X\ncong Y$, we say that $Y\leq_{\Delta} X$ is \textit{minimal}, if, for any $U\in D^b(Q)$ satisfying
 $$Y\leq_{\Delta} U\leq_{\Delta} X,$$
we have either $U\cong Y$ or $U\cong X$.

A direct consequence of Proposition \ref{prop:DegOrder} is the following.
\begin{cor}\label{cor:Min2Ext}
If $X\leq_{\Delta} Y$ is a minimal degeneration, then it is an ext-degeneration.
\end{cor}

\begin{lem}\label{lem:MinExtTriangle}
  If $X\le_{\mathrm{ext}} Y$ is minimal, then there exist two indecomposable direct summands  $Y_1$, $Y_2$ of $Y$ with a non-split triangle $Y_1\to E\to Y_2\to Y_1[1]$ such that 
     $$X\cong E\oplus \bigoplus_{j\ne 1,2}Y_j.$$
\begin{proof}
  By definition, there exists a decomposition $Y\cong Y'\oplus Y''$ fitting into a non-split triangle
  \[Y'\to X\to Y''\to Y'[1].\]
  We denote by $[\xi]$ the class of the above triangle in $\Ext^1(Y'',Y')$. Let $Y'\cong \bigoplus_i Y'_i$ be the decomposition into indecomposable objects. Note that 
   $$\Ext^1(Y'',Y')\cong \bigoplus_i \Ext^1(Y'',Y'_i).$$
  There exists at least an index $i$ such that $(\rho_i)_*[\xi]\ne 0$, where $\rho_i: Y'\to Y'_i$ is the canonical projection. Setting $Y^-=\bigoplus_{k\ne i}Y'_k$, we obtain a commutative diagram of triangles:
  \[\begin{tikzcd}
   &Y^- \arrow[r,equal]\arrow[d] &Y^-\arrow[d]\\
   {[\xi]}: &Y'\arrow[d,"\rho_2"]\arrow[r] &X\arrow[r,]\arrow[d] &Y''\arrow[r]\arrow[d,equal] &X[1]\arrow[d]\\
   {(\rho_{i})_*[\xi]}: & Y'_i\arrow[r]\arrow[d] &E\arrow[d]\arrow[r] &Y''\arrow[r] &X_i[1]\\
    &Y^-[1]\arrow[r,equal] &Y^-[1].\\
  \end{tikzcd}\]
 Therefore, we have
 \[Y^-\oplus Y'_i\oplus Y''\le_{\Delta} Y^-\oplus E\le_{\Delta} X.\]
 Since $[\xi']:=(\rho_i)_*[\xi]\ne 0$, we have $E\not\cong Y'_i\oplus Y''$. Since that $X\le_{\mathrm{ext}}Y$ is minimal, we can deduce that $ X\cong Y^-\oplus E$ and that $Y'_i\oplus Y''\le_{\Delta}E$ is minimal. Similarly,  consider the pullback of $[\xi']$ along the canonical inclusion $\iota_j:Y''_j\to Y''$ such that $\iota_j^*[\xi']\ne 0$. Let $E_{ij}$ be the middle term of $\iota_j^*[\xi']$. Applying the same arguments to $([\xi'],\iota_j^*[\xi'])$, we obtain $E\cong E_{ij}\oplus Y^-$ with $Y^{-}\oplus Y''_j\cong Y''$, which implies the desired result.
\end{proof}
\end{lem}

Note that the converse of the above Lemma does not hold in general, as the following example shows. 

\begin{ex}
For any vertex $i$ of $Q$, let $S_i$ be simple modules in $\mathrm{mod}(\mathbf{k}A_3)$, and $I_i$ (resp. $P_i$)  be the corresponding injective (resp. projective) module. Then the degeneration
\[I_2\le_{\Delta}S_2\oplus I_2\oplus S_2[1]\]
is given by a non-split extension of $S_2$ by $S_2[1]$. On the other hand, we also have
\[I_2\le_{\Delta} P_1\oplus S_3\le_{\Delta} P_1\oplus P_3\oplus S_2[1]\le_{\Delta}S_2\oplus I_2\oplus S_2[1]. \]
\end{ex}

\medskip

If $X\le_{\Delta}Y$ is not minimal, we would like to decompose it into a chain of minimal degenerations by using Proposition \ref{prop:DegOrder}. However, we need to show that the process stops and that we obtain a chain of finite length. In order to address this point, we will need the following well-known result (we provide a proof for the convenience of the reader).

\begin{lem}\label{lem:FiniteExt}
For any $X,Y\in D^b(Q)$, the set 
    \[\sE xt(X,Y)=\{L\in \mathrm{Iso}D^b(Q)|~X\to L\to Y\to X[1] \text{ is a triangle}\}\]
    is finite.
\begin{proof}
For any indecomposable object $Z_k=V(j_k,s_k)$ ($k=1,2$), write $Z_k[1]=V(j_k',s'_k)$. It can be seen from Happel's Theorem \ref{Happel's Thm} that the difference
  $$d=s'_2-s'_2=s'_1-s_1,$$
is constant. Moreover, for each $(i_0,p_0)\in \widehat{I}$, the full subcategory $\sH_{p_0}$ of $D^b(Q)$ whose indecomposable objects consists of $\{V(i,p)|~p_0\le p< p_0+d\}$ is an heart of $D^b(Q)$. In particular, $\sH_{p_0}\simeq \mathrm{mod}\mathbf{k}Q$.

  If $p_{min}(Y)\le p_{min}(X)$,  write $Y=\bigoplus_{k=1}^s V(i_k,p_k)$. Denote $$Y'=\bigoplus_{k,p_k\le p_{min}(X)}V(i_k,p_k),\qquad Y\cong Y'\oplus Y''.$$
  Since $\Ext^1(Y',X)=0$, we must have
  \[\sE xt(X,Y)=\{L'\oplus Y'|~L'\in \sE xt(X,Y'')\}.\]
  Therefore, we can assume $p_{min}(X)<p_{min}(Y)$. Similarly, if \[p_{max}(X)\ge p_{max}(Y)=p_{max}(Y''),\] by defining $X'$, $X''$ as before, we obtain
  \[\sE xt(X,Y)=\{L''\oplus X''\oplus Y'|~L''\in \sE xt(X',Y'')\}.\]
  Without loss of generality, we can assume 
  \[p_{min}(X)<p_{min}(Y),~ p_{max}(X)<p_{max}(Y).\]
  Thus, for any $Z\in \sE(X,Y)$, we have
  \begin{equation}\label{eq:ExtDistance}
       p_{min}(X)\le p_{min}(Z)\le p_{max}(Z)\le p_{max}(Y).
  \end{equation}

  We prove the statement by induction on $p_{max}(Y)-p_{min}(X)$. By Lemma \ref{lem:ExtDeg1}, for any $Z\in \sE xt(X,Y)$, there exists $Z'$ such that $(X\oplus Y)^{max}\cong Z^{X\oplus Y,max}\oplus Z'$ and $Z^{X\oplus Y,-}\in \sE xt(X,Y^-\oplus Z')$. Hence, we have 
   \[\begin{aligned}
      &\sE xt(X,Y)\\
      =&\{Z'\oplus Z''\in \mathrm{Iso}D^b(Q)|~Z''\in \sE xt(X,Y^-\oplus Z')\text{ such that }p_{max}(Z'')<p_{max}(Y), ~|Z'|\subset |Y^{max}|\}.
  \end{aligned}\]
  Moreover, since by the assumption (\ref{eq:ExtDistance}) we also have 
  $$p_{min}(Z'')\ge p_{min}((X\oplus Y)^+\oplus Z')=p_{min}(X).$$
  If $p_{max}(Y)-p_{min}(X)\le d$, 
  it follows that for any $Z''\in \sE xt(X,Y^-\oplus Z')$, we have $Z''\in \sH_{p_{min}(X)}$. Since there are only finitely many isoclasses with a fixed dimension vector in $\mathrm{mod}\kQ$, we deduce that $\sE xt(X,Y)$ is finite, by noting that $\sH_{p_{min}(X)}\simeq \mathrm{mod}\kQ$ and $\{Z'\in \mathrm{Iso}D^b(Q)|~|Z'|\subseteq |Y|\}$ are finite.

  When $p_{max}(Y)-p_{min}(X)>  d$,  let $\iota:Y^{max}\to Y$ be the canonical embedding. For any triangle $[\xi]:X\to Z\to Y\to X[1]$, consider the following commutative diagram of triangles: 
  $$\begin{tikzcd}
     & &Y^-[-1]\arrow[r,equal] \arrow[d] &Y^-[-1]\arrow[d]\\
    {\iota^*[\xi]}:  &X\arrow[d,equal]\arrow[r] &Z'\arrow[d]\arrow[r] &Y^{max}\arrow[r]\arrow[d,"\iota"] &X[1]\arrow[d,equal]\\
     {[\xi]}: &X\arrow[r] &Z\arrow[d]\arrow[r] &Y\arrow[d]\arrow[r] &Y'[1].\\
     &  &Y^-\arrow[r,equal] &Y^-
    \end{tikzcd}$$
  Hence
  \[\sE xt(X,Y)\subset \bigcup_{Z'\in \sE xt(X,Y^{max})}\sE xt(Z',Y^-).\]
  Moreover, for any $Z'\in \sE xt(X,Y^{max})$,  we have 
  $$p_{min}(Z')\ge \min\{p_{min}(X),p_{min}(Y^{max})\}=p_{min}(X),$$
  as $p_{min}(Y^{max})=p_{max}(Y)>p_{min}(X)$, so
  $$p_{max}(Y^-)-p_{min}(Z')<p_{max}(Y)-p_{min}(X).$$
  By induction hypothesis, $\sE xt(Z',Y^-)$ is finite. Write $X=\bigoplus_k V(i'_k,p'_k)$ and denote
  \[L_{+}:=\bigoplus_{k,~p_k\ge p_{max}(Y)-d}V(i_k,p_k),\qquad L_{-}:=\bigoplus_{k,~p_k< p_{max}(Y)-d}V(i_k,p_k),\]
  then $X\cong L_+\oplus L_-$. Since $\Ext^1(Y^{max},L_-)=0$, we have 
  \[\sE xt(X,Y^{max})=\{Z''\oplus L_-|~Z''\in \sE xt(L_+,Y^{max})\}.\]
  Since $p_{max}(Y)-p_{min}(L_+)\le d$, by the previous arguments, $\sE xt(X,Y^{max})$ is a finite set. Hence, the set $\sE xt(X,Y)$ is finite.

\end{proof}
\end{lem}

\begin{prop}\label{prop:FiniteDeg}
  For any $Y\in D^b(Q)$, the set $\mathrm{Deg}(Y)=\{X\in \mathrm{Iso}D^b(Q)|~X\le_{\Delta}Y\}$
  is finite.
\begin{proof}
  We proceed the proof by induction on $t=p_{max}(Y)-p_{min}(Y)$. If $t=1$, then objects in $|Y|$ are perpendicular to each other. Hence $\mathrm{Deg}(Y)=\{Y\}$. For $t\ge 1$, write

  \[ \Sigma_1= \{X\in \mathrm{Deg}(Y)|~p_{max}(X)<p_{max}(Y)\},\]
  
  and 
  \[\Sigma_2= \{X\in \mathrm{Deg}(Y)|~p_{max}(X)=p_{max}(Y)\}.\]
  Then, by Lemma \ref{lem:Boundary},
  \[
  \begin{aligned}
    &\mathrm{Deg}(Y)=\Sigma_1\sqcup \Sigma_2.
  \end{aligned}\]
  Moreover, by Corollary \ref{cor:ExtDeg},
  $$\Sigma_2\subset \bigcup_{|Z|\subsetneq |Y^{max}|}\{X'\oplus (Y^{max}/Z)\in \mathrm{Iso}D^b(Q)|~X'\in \mathrm{Deg}\bigl(Y^-\oplus Z\bigr)\},$$
  where $Y^{max}/Z$ is the quotient of the canonical embedding $Z\hookrightarrow Y^{max}$.
  If $\#|Y^{max}|=1$, then 
  $$\Sigma_2=\{X'\oplus Y^{max}|~X'\in \mathrm{Deg}(Y^-)\},$$
  which is finite by the induction hypothesis. For any $X\in \Sigma_1$, by Proposition \ref{prop:DegOrder}, there must exists $M\in \Sigma_2$ such that $X\le_{\mathrm{ext}}M$. Denote $\Delta(M)=\{X\in \mathrm{Iso}D^b(Q)|~X\le_{ext}M\}$. Hence,
  \[\Sigma_1\subset \bigcup_{M\in \Sigma_2} \Delta(M),\]
  which is finite since there are only finitely many decompositions $M'\oplus M''$ of $M$ and $\sE xt(M',M'')$ is finite by Lemma \ref{lem:FiniteExt}. Therefore, $\mathrm{Deg}(Y)$ is finite. If $\#|Y^{max}|\ge 2$, then we can perform induction on $\#|Y^{max}|$ to deduce that $\Sigma_2$ is finite.  Then we can show that $\Sigma_1$ is finite as before, which implies the desired result.
  
\end{proof}
\end{prop}

\begin{cor}\label{cor:ChainMinDeg}
  For any $X\le_{\Delta}Y$, there exists a chain of minimal degenerations 
  \[X=N_0\le_{\mathrm{ext}} N_1\le_{\mathrm{ext}}\cdots\le_{\mathrm{ext}}N_l=Y.\]
\end{cor}

\medskip
\subsection{Main theorem}
  To show that the injection 
  \[V(-):\{n\le m|~n\in \sM^+\}\hookrightarrow \{N\in \mathrm{Iso}D^b(Q)\ |\ N\le_{\Delta}V(m)\}\] is surjective, it is enough to show that, for any  minimal degeneration $M_{i-1}\le_{\mathrm{ext}}M_i$, we have $m^{(i-1)}\le m^{(i)}$. Here $m^{(i)}\in \sM^+$ is the monomial satisfying  $V(m^{(i)})\cong M_{i}$.

  The mesh relations of the Auslander--Reiten quiver tell us that each Auslander--Reiten triangle gives rise to a degeneration respect to the Nakajima's order. Namely, let 
  \[\tau Y\to Z\to Y\to X[1]\]
  be an Auslander--Reiten triangle, then we have $z\le y'y$, where $V(y')=\tau Y$, $V(y)=Y$ and $V(z)=Z$. In the following, we generalize this result to arbitrary triangles where the first and the third terms are indecomposable objects.

\begin{lem}\label{lem:Triangle2Order}
  Let   $[\xi]: V(Y_{i,p})\to V(n)\to V(Y_{j,s})\to V(Y_{i,p})[1]$ be a triangle. Then $n\le Y_{i,p}Y_{j,s}$.
\begin{proof}
  If $[\xi]=0$, then the triangle splits and $V(n)\cong V(Y_{i,p})\oplus V(Y_{j,s})$, which implies $n=Y_{i,p}Y_{j,s}$; Otherwise, denote $\sE=\sE xt(V(Y_{i,p}),V(Y_{j,s}))$. It follows from Equation (\ref{eq:barQTStd}) that
   \begin{align*}
       \Phi(\overline{M_t(Y_{i,p}Y_{j,s})}) =&v^{-\alpha(Y_{i,p}Y_{j,s})+1}(v-v^{-1})^2u_{V(Y_{i,p})} u_{V(Y_{j,s})}\\
       =&(v-v^{-1})^2v^{-\alpha(Y_{i,p}Y_{j,s})+1}\sum_{N\in \sE} F_{V(Y_{i,p}),V(Y_{j,s})}^N {u}_N
    \end{align*}
  Moreover, for any $n'\in \sM^+$, by Equation (\ref{qtStd}),  $\Phi(M_t(n'))=\beta(n')u_{V(n')}$, for some positive coefficient $\beta(n')$,  which implies
  \[\Phi(\overline{M_t(Y_{i,p}Y_{j,s})})=\sum_{n',\ V(n')\in \sE} b(n')\Phi(M_t(n')),\]
  for some positive coefficients $b(n')$.
  On the other hand, by Equation (\ref{barInv}), we also have
  \[\Phi(\overline{M_t(Y_{i,p}Y_{j,s})})=\sum_{n'\le m,\ n'\in \sM^+} a(n') \Phi(M_t(n')).\]
  Since $(\Phi(M_t(n'))_{n'\in \sM^+}$ forms a basis of $\sD\sH(Q)$ and $b(n')\ne 0$ for all $n'$ such that $V(n')\in \sE$, we must have $n'\le Y_{i,p}Y_{j,s}$ for any $n'$. Therefore, since $V(n)\in \sE$,  we deduce $n\le Y_{i,p}Y_{j,s}$.
\end{proof}
\end{lem}

Combining the above lemma with Corollary \ref{cor:Min2Ext} Lemma \ref{lem:MinExtTriangle}, we have the following
\begin{prop}\label{prop:degOrdtoNKJMOrd}
 For any minimal degeneration $V(x)\le_{\Delta}V(y)$, we have $x\le y$.
\end{prop}

Our main result follows.
\begin{teo}\label{Thm:Correspondence}
The map $V(-)$ is an isomorphism of partially ordered sets:
\[ V(-): \{m\in \mathcal{M}^+\} \xrightarrow{\sim} \{ N\in \mathrm{Iso}D^b(Q)\},\quad m\mapsto V(m).\]
In particular, for any dominant monomials $m$ and $m'$, 
    \[ m'\leq m \quad \text{ if and only if } \quad V(m')\leq_{\Delta} V(m).\]

\end{teo}

\begin{proof}
Let $m'$ and $m$ be dominant monomials. Assume first that $m'\leq m$. Let $W$ and $V$ be respectively a $\widehat{I}$-graded $\mathbb{C}$-vector space and a $\widehat{I}'$-graded $\mathbb{C}$-vector space such that $m=m(0,W)$ and $m'=m(V,W)$. Then, by Proposition \ref{prop_qin}, the regular stratum $\mathfrak{M}^\mathrm{reg}_0(V,W)$ is not empty. Therefor, by Proposition \ref{pro_KelScher_bijection}, this implies that $V(m')\leq V(m)$.

 For any $N\le_{\Delta}M$, denote $M=V(m)$ and $N=V(m')$ for some $m',m\in \sM^+$. By Corollary \ref{cor:ChainMinDeg}, we have a chain of minimal extensions of finite length:
 \[N=M_1\le_{\mathrm{ext}}M_2\le_{\mathrm{ext}}\cdots \le_{\mathrm{ext}} M_{s-1}\le_{\mathrm{ext}}M_s=M.\]
Write $V(m^{(i)})=M_i$.  By Proposition \ref{prop:degOrdtoNKJMOrd}, we deduce that
\[m'\le m^{(1)}\le \cdots \le m^{(s)}=m,\]
as we wanted.
\end{proof}

\begin{cor}
For any $\widehat{I}$-graded $\mathbb{C}$-vector space $W$, the following partially ordered sets are isomorphic:
\[\begin{tikzpicture}[
   hookright/.style={{Hooks[right]}->},
   hookleft/.style={{Hooks[left]}->},]
\node (M) at (0,0) {$\{ \ssM_0^{\bullet\mathrm{reg}}(V,W)\ |\ \ssM_0^{\bullet\mathrm{reg}}(V,W)\neq \emptyset\}$};
\node (m) at (8,0) {$\{m'\in \mathcal{M}^+\ |\ m'\leq m(0,W)\}$};
\node (D) at (4,-2.5) {$\{N\in \mathrm{Iso}(D^b(Q))\ |\ N\leq V(m(V,W))\}$};
\draw[->] (M) --  
(m) node[midway, above] {$\sim$}
node[midway, below] {$m(-,-)$};
\draw[->] (M) --  
(D) node[midway, below, sloped] {$\phi$} node[midway, above, sloped] {$\sim$};
\draw[->] (m) --  
(D) node[midway, below, sloped] {$V(-)$} node[midway, above, sloped] {$\sim$};;
\end{tikzpicture}
\]
\end{cor}

\subsection{Degenerations in type $A_n$}
In this section, we assume that the quiver $Q$ is of type $A_n$ and, for any minimal degeneration $V(m')\leq_{\Delta} V(m)$, we provide a formula for the monomial $A$ in the variables $A_{i,p+1}$, $(i,p)\in \widehat{I}$, such that 
\[ m'=mA^{-1}.\]
The repetitive quiver $\widehat{Q}$ is locally as follows:
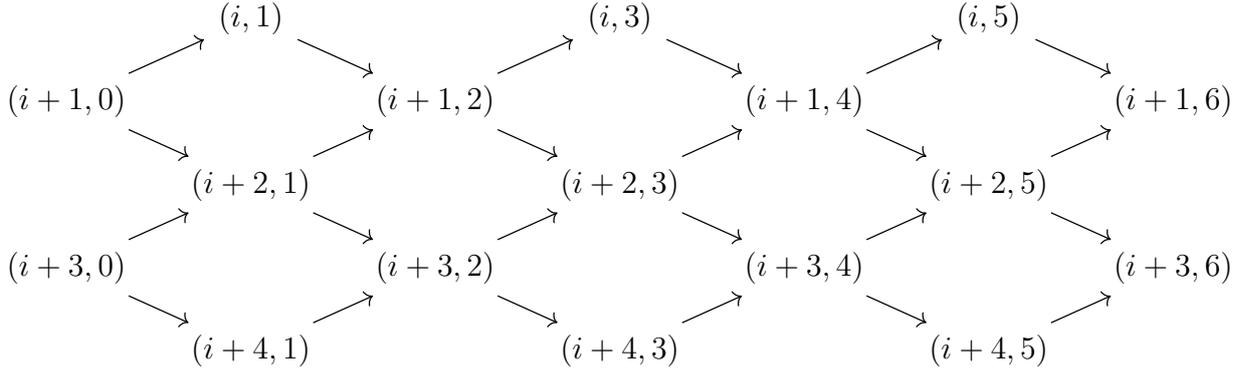
\[\begin{tikzcd}[sep=small]
    &(i,1)\arrow[rd] & &(i,3)\arrow[rd] & &(i,5)\arrow[rd] & \\
 (i+1,0)\arrow[rd]\arrow[ru] & &(i+1,2)\arrow[rd]\arrow[ru] & &(i+1,4)\arrow[rd]\arrow[ru] & &(i+1,6) \\
  &(i+2,1)\arrow[rd]\arrow[ru] & &(i+2,3)\arrow[rd]\arrow[ru] & &(i+2,5)\arrow[rd]\arrow[ru] &\\
  (i+3,0)\arrow[rd]\arrow[ru] & &(i+3,2)\arrow[rd]\arrow[ru] & &(i+3,4)\arrow[rd]\arrow[ru] & &(i+3,6) \\
    &(i+4,1)\arrow[ru] & &(i+4,3)\arrow[ru] & &(i+4,5)\arrow[ru] & \\
\end{tikzcd}\]
Note that, for any indecomposable objects $X,Y\in D^b(Q)$, we have
$$\dim \Ext^1_{D^b(Q)}(Y,X)\le 1\qquad\text{and}\qquad \dim \Hom_{D^b(Q)}(X,Y)\le 1.$$
Hence, $\Ext^1_{D^b(Q)}(Y,X)\ne 0$, the non-trivial extension $Z$ of $X$ by $Y$ consists of at most two indecomposable objects (this can be deduced by applying $\R\Hom_{D^b(Q)}(X,-)$ to the corresponding triangle). 
Recall from Section \ref{subsect_Happel} Happel's functor $V(-)$.
For convenience of notation, for any $i=0$ and $p\in \Z$, we set $V(i,p)=0$.
Let $(i,p)$, $(i',p')$ be in $\widehat{I}$ and, for $k=1,2$, take $i_k$ in $I\sqcup\{0,n+1\}$ and $p_k$ in $\mathbb{Z}$ with $p_1\leq p_2$.
Then, by \cite[Section 3.1]{schiffler2014}, there exists a non-split triangle 
  $$V(i,p)\to V(i_1,p_1)\oplus V(i_2,p_2)\to V(i',p')\to V(i,p)[1]$$ if and only if there exist positive integers $a\ge b \ge 1$ satisfying one of the following conditions:

\begin{itemize}
  \item[(C1)]\label{Cond1} $i-b$, $i+a\in \{0,1,\cdots,n,n+1\}$, and $(i_1,p_1)=(i-b,p+b)$, $(i_2,p_2)=(i+a,p+a)$, and $(i',p')=(i+a-b,p+a+b)$.
  
  \item [(C2)]\label{Cond2} $i+b$, $i-a\in \{0,1,\cdots,n,n+1\}$, and $(i_1,p_1)=(i+b,p+b)$, $(i_2,p_2)=(i-a,p+a),$ and $(i',p')=(i+b-a,p+a+b)$.
    
\end{itemize}

Graphically, at the level of the repetitive quiver, these two conditions translates into  configurations where the indecomposable objects of the triangle sit at the vertices of a parallelogram,  respectively of the form:
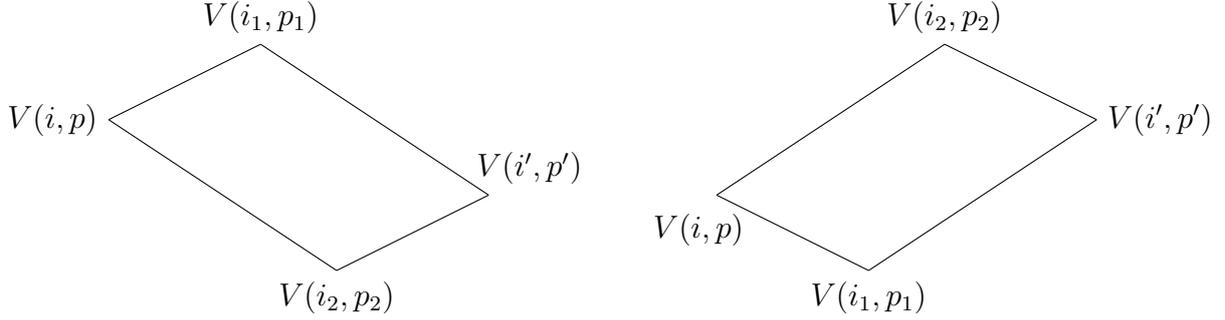
\[
  \begin{tikzpicture}
    \draw (0,0)--(2,1);
    \draw (2,1)--(5,-1);
    \draw (0,0)--(3,-2);
    \draw (3,-2)--(5,-1);
     \draw (8,-1)--(10,-2);
    \draw (10,-2)--(13,0);
    \draw (8,-1)--(11,1);
    \draw (11,1)--(13,0);
    \node[anchor=east] at (0,0){$V(i,p)$}; 
    \node[anchor=south] at (2,1) {$V(i_1,p_1)$};
    \node[anchor=north] at (3,-2) {$V(i_2,p_2)$};
    \node[anchor=south west] at (4.7,-1) {$V(i',p')$};
     \node[anchor=north east] at (8.5,-1.1){$V(i,p)$}; 
    \node[anchor=north] at (10,-2) {$V(i_1,p_1)$};
    \node[anchor=south] at (11,1) {$V(i_2,p_2)$};
    \node[anchor=west] at (13,0) {$V(i',p')$};
  \end{tikzpicture}
\]

By Lemma \ref{lem:Triangle2Order}, for any $V(Y_{i,p})$ and $V(Y_{j,s})$ such that there exists a triangle 
$$V(Y_{i,p})\to V(m')\to V(Y_{j,s})\to V(Y_{i,p})[1],$$
we have $m'\le Y_{i,p}Y_{j,s}$, i.e. $m'=Y_{i,p}Y_{j,s}A^{-1}$ with $A$ a product of certain $A_{i',p'}$ ($(i',p')\in \widehat{I}'$). In the following proposition we give an explicit description of the monomial $A$.
\begin{prop}
  For any ext-degeneration $m'\le_{\Delta} Y_{i,p}Y_{j,s}$ there exists $1\le b\le a$ satisfying Condition $(C1)$ or $(C2)$. Moreover,
  \[m'=Y_{i,p}Y_{j,s}\prod_{0\le r<a}\ \prod_{0\le l<b}A^{-1}_{i+r-l,p+r+l+1},\]
  if $(a,b)$ satisfies Condition $(C1)$ and
  \[m'=Y_{i,p}Y_{j,s}\prod_{0\le r<a}\ \prod_{0\le l<b}A^{-1}_{i-r+l,p+r+l+1},\]
 if $(a,b)$ satisfies Condition $(C2)$.
\end{prop}
\begin{proof}
The first statement  follows from the above discussion.
We only need to show the second statement holds if $(a,b)$ satisfies Condition $(C1)$, as the other case can be proved in the similar way.

The base case $a=1$ and $b=1$ is just a mesh. Assume that $b=1$ and that the statement holds up to $a-1\geq 1$. We want to prove the statement for $(a,1)$.  By condition $(C1)$, we have $m'=Y_{i_1,p_1}Y_{i_2,p_2}$, where $(i_1,p_1)=(i-1,p+1)$ and $(i_2,p_2)=(i+a,p+a)$. Set $(i',p')=(i+a-2,p+a)$ and $(i_3,p_3)=(i+a-1,p+a-1)$.
Consider the parallelogram formed by $\bigl((i,p),(i_1,p_1),(i_2,p_2),(j,s)\bigr)$ as follows:
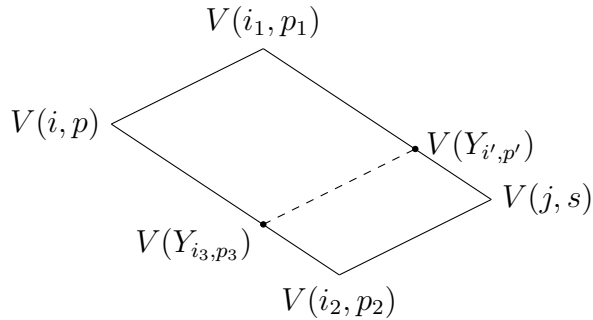
\[ \begin{tikzpicture}
    \draw (0,0)--(2,1);
    \draw (2,1)--(5,-1);
    \draw (0,0)--(3,-2);
    \draw (3,-2)--(5,-1);
    \draw (2,-1.33)--(4,-0.33) [dashed];
    \draw[fill=black] (4,-0.33) circle (1pt)  (2,-1.33) circle (1pt);
    \node[anchor=east] at (0,0){$V(i,p)$}; 
    \node[anchor=south] at (2,1) {$V(i_1,p_1)$};
    \node[anchor=north] at (3,-2) {$V(i_2,p_2)$};
    \node[anchor=west] at (5,-1) {$V(j,s)$};
    \node[anchor=west] at (4,-0.3) {$V(Y_{i',p'})$};
    \node[anchor=east] at (2,-1.6) {$V(Y_{i_3,p_3})$};
  \end{tikzpicture}\]
   Note that there are two sub-parallelogram: the upper and the lower one. Since  $(j,s)=(i+a-1,p+a+1)$, by induction hypothesis,
   \[Y_{i_1,p_1}Y_{i_3,p_3}=Y_{i,p}Y_{i',p'}\prod_{r=0}^{a-2}A_{i+r,p+r+1},\qquad Y_{i_2,p_2}Y_{i',p'}=Y_{i_3,p_3}Y_{j,s}A_{i+a-1,p+a}.\]
   Hence we have 
   \[Y_{i_1,p_1}Y_{i_3,p_3}Y_{i_2,p_2}=Y_{i,p}Y_{j,s}Y_{i_3,p_3}\prod_{r=0}^{a-1}A_{i+r,p+r+1},\]
   which is the desired result.  

   If $b\ge 2$, set $(i_3,p_3)=(i-b+1,p+b-1)$ and  $(i',p')=(i+a-b+1,p+a+b-1)$. As above, we obtain two sub-parallelograms as shown in the following figure: one formed by $\bigl((i,p),(i_1,p_1),(i_3,p_3),(i',p')\bigr)$, and  one formed by $\bigl((i_3,p_3),(i',p'),(i_2,p_2),(j,s)\bigr)$. 
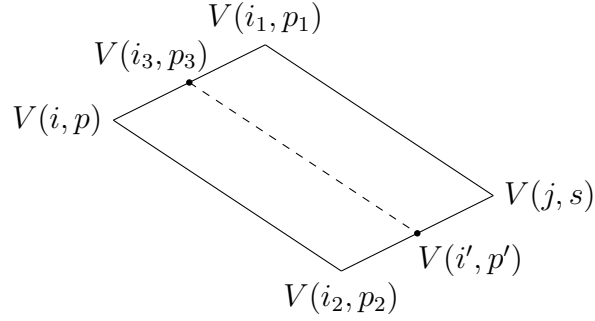
   \[\begin{tikzpicture}
    \draw (0,0)--(2,1);
    \draw (2,1)--(5,-1);
    \draw (0,0)--(3,-2);
    \draw (3,-2)--(5,-1);
    \draw (1,0.5)--(4,-1.5) [dashed];
    \draw[fill=black] (1,0.5) circle (1pt)  (4,-1.5) circle (1pt);
    \node[anchor=east] at (0,0){$V(i,p)$}; 
    \node[anchor=south] at (2,1) {$V(i_1,p_1)$};
    \node[anchor=north] at (3,-2) {$V(i_2,p_2)$};
    \node[anchor=west] at (5,-1) {$V(j,s)$};
    \node[anchor=south] at (0.5,0.5) {$V(i_3,p_3)$};
    \node[anchor=north] at (4.7,-1.5) {$V(i',p')$};
  \end{tikzpicture}\]
  Moreover, the associated pairs are $(a,b-1)$ and $(a,1)$ respectively, and they satisfy Condition $(C1)$. By induction hypothesis, we have
   \[Y_{i_2,p_2}Y_{i_3,p_3}=Y_{i,p}Y_{i',p'}\prod_{r=0}^{a-1}\ \prod_{l=0}^{b-2}A_{i+r-l,p+r+l+1},\qquad Y_{i_1,p_1}Y_{i',p'}=Y_{i_3,p_3}Y_{j,s}\prod_{r=0}^{a-1}A_{i+r-b+1,p+r+b}.\]
   Hence we have 
   \[Y_{i_1,p_1}Y_{i_2,p_2}Y_{i_3,p_3}=Y_{i,p}Y_{j,s}Y_{i_3,p_3}\prod_{r=0}^{a-1}\ \prod_{l=0}^{b-1}A_{i+r-l,p+r+l+1},\]
   which completes the proof.

\end{proof}

\medskip
\section*{Acknowledgement}
The authors would like to thank Ryo Fujita, Bernhard Keller, David Hernandez, and Grzegorz Zwara for many fruitful discussions and useful insights.
\medskip


\providecommand{\bysame}{\leavevmode\hbox to3em{\hrulefill}\thinspace}
\providecommand{\MR}{\relax\ifhmode\unskip\space\fi MR }
\providecommand{\MRhref}[2]{%
  \href{http://www.ams.org/mathscinet-getitem?mr=#1}{#2}
}
\providecommand{\href}[2]{#2}

\end{document}